\documentclass[11pt,a4paper]{article}

\usepackage{epsf,epsfig,amsfonts,amsgen,amsmath,amstext,amsbsy,amsopn,amsthm,amsfonts,amssymb,amscd,arydshln
}
\usepackage[numbers,sort&compress]{natbib}
\usepackage{color,pict2e}

\newtheorem{thm}{Theorem}[section]

\newtheorem{lemma}{Lemma}[section]
\newtheorem{prob}{Problem}
\newtheorem{remark}{Remark}

\begin{document}
	
	\baselineskip 16pt
	\newcommand{\la}{\lambda}
	\newcommand{\si}{\sigma}
	\newcommand{\ol}{1-\lambda}
	\newcommand{\be}{\begin{equation}}
	\newcommand{\ee}{\end{equation}}
	\newcommand{\bea}{\begin{eqnarray}}
	\newcommand{\eea}{\end{eqnarray}}
	\newcommand{\bL}{\b{\textit{L}}}
	\newcommand{\bN}{\b{\textit{N}}}
	\newcommand{\bB}{\b{\textit{B}}}

\title{\bf\Large Graphs with three distinct distance eigenvalues}

\date{}

\author{
	Yuke Zhang  ~and Huiqiu Lin\footnote{Supported
		by the National Natural Science Foundation of China (Nos. 11771141, 12011530064 and 11871015).\quad\quad\quad E-mail: huiqiulin@126.com (H.Q. Lin),~zhang\_yk1029@163.com (Y.K. Zhang)}\\
	{\footnotesize School of Mathematics,
		East China University of Science and Technology,}\\
	{\footnotesize Shanghai 200237, P.R.~China}
}
\maketitle
\vspace{-9mm}

\maketitle
	
\begin{abstract}
In this paper, some special distance spectral properties of graphs are considered. Concretely, we recursively construct an infinite family of trees with distance eigenvalue $-1$, and determine all $\{C_3,C_4\}$-free connected graphs with three distinct distance eigenvalues of which the smallest one is equal to $-3$, which partially answers a problem posed by Koolen, Hayat and Iqbal [Linear Algebra Appl. 505 (2016) 97--108]. Furthermore, we characterize all trees with three distinct distance eigenvalues.

\vspace{5mm}
\vspace{3mm}

\end{abstract}

\baselineskip=0.30in

\section{Introduction}
 Let $G=(V(G), E(G))$ be a connected graph with vertex set
$V(G)=\{v_{1}, v_{2}, \ldots,$ $v_{n}\}$ and edge set $E(G)$.
The \emph{distance}  between $v_i$ and $v_j$, denoted by $d_{i,j}$, is the length of a shortest path from $v_i$ to $v_j$. The \emph{distance matrix} of $G$,
denoted by $D(G)$, is the $n\times n$ real symmetric matrix whose $(i,\,j)$-entry is $d_{i,j}$.
The multiset  $\{\lambda_1^{m_1}(D(G)),\ldots,\lambda_k^{m_k}(D(G))\}$ is called the \emph{distance spectrum} of $G$, and denoted by $Spec(D(G))$,  where  $\lambda_1(D(G))>\cdots>\lambda_k(D(G))$ are all distinct eigenvalues of $D(G)$, and $m_i$ is the multiplicity of $\lambda_i(D(G))$.

The original study of distance eigenvalues was launched around the trees and can be traced back to the paper of Graham and
Pollack \cite{GP} in which they obtained a very interesting and excellent result that the
determinant of the distance matrix of a tree depends only on the number of vertices, and not
on the structure of the tree.
\begin{thm}[\cite{GP}]\label{dett}
	Let $T$ be a tree with $n$ vertices and $ D(T) $ be its distance matrix. Then
$$det(D(T))=(-1)^{n-1}(n-1)2^{n-2}. $$
\end{thm}
This insightful result made distance matrix spectral properties a research subject of
great interest and motivated lots of researches on distance matrix of a tree such as the inverse \cite{GL} and the characteristic polynomial \cite{EGG,GL,HMG,GP}.
Another remarkable result is from Merris \cite{merris}, which stated a relation between the distance eigenvalues and Laplacain eigenvalues of a tree.
\begin{thm}[\rm\cite{merris}]\label{mer}
	Let $G$ be a tree of order $n$. Let
	$\lambda_1(D(G))\geq\cdots\geq \lambda_n(D(G))$
	be the eigenvalues of $D(G)$ and let
	$\mu_1\geq \mu_2\geq\cdots\geq\mu_{n-1}\geq0$
	be the Laplacain eigenvalues of $ G $. Then
	$$
	0>\frac{-2}{\mu_1}>\lambda_2(D(G))\geq\frac{-2}{\mu_2}\geq\cdots\geq\lambda_{n-1}(D(G))\geq\frac{-2}{\mu_{n-1}}\geq\lambda_{n}(D(G)).
	$$
\end{thm}
Since then, the study of distance eigenvalues of a graph has received much more attention.
These two theorems
proved by Graham and Pollack \cite{GP} and Merris \cite{merris} are instructive and also play important roles in our follow-up proofs.
For more results on the distance matrix and its spectral properties, we refer the reader to two  surveys \cite{AH,LSXZ}.


It is well-known that the distance spectrum
of any graph contains at least two distinct eigenvalues since the diagonal entries of the distance matrix are all 0. So another problem concerning graph eigenvalues is characterizing graphs with some distinct distance eigenvalues.
This problem can be considered into two types. One is determining graphs with some prescribed distance eigenvalues. The other is considering graphs with some distinst distance eigenvalues but without given values.

For the former,
Lu, Huang and Huang \cite{LHH} determined all graphs whose distance matrices have exactly two eigenvalues (counting multiplicity) different from $ -1 $ and $ -3 $ and some other related results see \cite{HHL,zhang}. Besides, the results in \cite{LHL} and \cite{LL} stated that the multiplicities of distance eigenvalues $ -1 $ and $-2$ in a threshold graph (a graph contains no induced $ C_4, P_4 $ or $ 2K_2 $) and a cograph (a graph contains no induced $ P_4 $), respectively. In general connected graphs, Li and Meng \cite{LM} characterized the graphs with certain multiplicity of the distance eigenvalue $ -2 $.
We focus on the following problem posed in \cite{Koolen}.
\begin{prob}[Problem 6.2 in \cite{Koolen}]\label{p1}
	Determine the connected graphs with three distinct distance eigenvalues
	$\lambda_1(D(G))>\lambda_2(D(G))>\lambda_3(D(G))$ such that $ \lambda_3(D(G))=-3 $.
\end{prob}
Let $ \mathcal{F} $ be a family of graphs. We say a graph $ G $ is $ \mathcal{F} $-free if it does not contain any $F\in \mathcal{F} $ as a subgraph.
 In this paper, we consider Problem \ref{p1} in $\{C_3,C_4\}$-free connected graphs.
\begin{thm}\label{t3}
	Let $ G $ be a $\{C_3,C_4\}$-free connected graph. Then $ G $ contains three distinct distance eigenvalues
	$\lambda_1(D(G))>\lambda_2(D(G))>\lambda_3(D(G))$ such that $ \lambda_3(D(G))=-3 $ if and only if $ G $ is the Petersen graph.
\end{thm}

For the latter,
Indulal \cite{83} showed that the complete graph is the
only graph that contains exactly two distinct distance eigenvalues.
Lin, Hong, Wang and Shu \cite{lhws} found two types of graphs with exactly three distinct distance eigenvalues.
Besides,
Koolen, Hayat and Iqbal \cite{Koolen} showed that the graphs with exactly three distinct distance eigenvalues fall into four classes. Based on these four classes \cite{Koolen}, we determine all trees with three distinct distance eigenvalues as an application of Theorem \ref{dett} and Theorem \ref{mer} in this paper.

\begin{thm}\label{3tree}
Any tree contains exactly three distinct distance eigenvalues if and only if it is a star.
	\end{thm}
The distance eigenvalues are closely linked to the structure of a graph. We consider the structure of a tree containing some specific value as its distance eigenvalue.
For all $ n $-vertex trees, it is known \cite{BRT,guo} that the number of the Laplacian eigenvalues distributed in $ [2, n] $ is at most $\lfloor\frac{n}{2}\rfloor$. We obtain an analog that the number of the distance eigenvalues distributed in $ [-1, 0) $ is at most $\lfloor\frac{n}{2}\rfloor$ by Theorem \ref{dett}. So it is  significant to characterize the trees with $ -1 $ as their distance eigenvalues.
In \cite{KT} and \cite{LHH}, authors showed that distance eigenvalue $ -1 $ can determine some special structure of a graph, respectively. However, these results cease to be effective in trees.

In this paper, we structure a series of trees with $-1$ as their distance eigenvalues, which also provide a sufficient condition for trees to have $-1$ as their distance eigenvalues.
Denote by $ \mathcal{T}(4n+2) $ the set of all trees with $ 4n+2 $ vertices obtained by adding an edge joining an endvertex of $ P_4 $ to a vertex of a tree in $ \mathcal{T}(4n-2) $ (see Fig. \ref{treeyou-1}). Note that the choices of the tree $ T\in \mathcal{T}(4n-2) $ and joinning vertex $ v\in T $ are arbitrary. Clearly, $ \mathcal{T}(2)=\{P_2\} $ and $ \mathcal{T}(6)=\{P_6\} $.

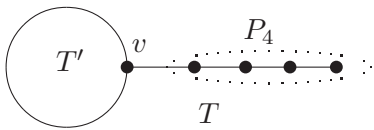
\begin{figure}[htp]
	\setlength{\unitlength}{1.2pt}
	\begin{center}
		\begin{picture}(113.8,26.1)
		\put(37.7,5.1){\circle*{4}}
		\qbezier(37.7,5.1)(37.7,-2.7)(32.2,-8.3)\qbezier(32.2,-8.3)(26.7,-13.8)(18.9,-13.8)\qbezier(18.9,-13.8)(11.0,-13.8)(5.5,-8.3)\qbezier(5.5,-8.3)(0.0,-2.7)(0.0,5.1)\qbezier(0.0,5.1)(-0.0,12.9)(5.5,18.4)\qbezier(5.5,18.4)(11.0,23.9)(18.8,23.9)\qbezier(18.8,23.9)(26.7,23.9)(32.2,18.4)\qbezier(32.2,18.4)(37.7,12.9)(37.7,5.1)
		\put(74.7,5.1){\circle*{4}}
		\put(58.7,5.1){\circle*{4}}
		\put(103.0,5.1){\circle*{4}}
		\qbezier(74.7,5.1)(88.5,5.1)(103.0,5.1)
		\put(88.5,5.1){\circle*{4}}
		\qbezier(58.7,5.1)(66.7,5.1)(74.7,5.1)
		\qbezier(58.7,5.1)(48.2,5.1)(37.7,5.1)
		\qbezier[2](113.8,5.1)(113.8,3.0)(104.5,1.5)\qbezier[6](104.5,1.5)(95.1,0.0)(81.9,0.0)\qbezier[6](81.9,0.0)(68.7,0.0)(59.4,1.5)\qbezier[2](59.4,1.5)(50.0,3.0)(50.0,5.1)\qbezier[2](50.0,5.1)(50.0,7.2)(59.4,8.7)\qbezier[6](59.4,8.7)(68.7,10.1)(81.9,10.2)\qbezier[6](81.9,10.2)(95.1,10.2)(104.5,8.7)\qbezier[2](104.5,8.7)(113.8,7.2)(113.8,5.1)\put(74.0,20.1){\makebox(0,0)[tl]{$P_4$}}
		\put(15.5,10.8){\makebox(0,0)[tl]{$T'$}}
		\put(39.2,13.9){\makebox(0,0)[tl]{$v$}}
		\put(60.0,-6.1){\makebox(0,0)[tl]{$T$}}
		\end{picture}
	\end{center}
	\caption{The tree $T\in \mathcal{T}(4n+2)$ where $T'\in \mathcal{T}(4n-2) $ and $v\in T'$.}\label{treeyou-1}
\end{figure}

\begin{thm}\label{-1}
	For any $n\in \mathbb{N}$ and tree $ T\in \mathcal{T}(4n+2) $, there always exists $ -1 $ as an eigenvalue of $ D(T) $.
\end{thm}
While it should be emphasized that it is not the unique structure to characterize trees with distance eigenvalues $ -1 $. We find a tree whose distance spectrum contains $ -1 $ but does not belong to $ \mathcal{T}(4n+2) $. Let $ S_{a,b} $ be the \textit{double star} obtained by attaching $ a $ and $ b $ pendant vertices to the
two end vertices of $ K_2 $, respectively.

\begin{remark}\label{sp}
	Let $ T $ be a tree with diameter at most three. The value $ -1 $ is an eigenvalue of $ D(T) $ if and only if $ T \cong P_2$ or $ T \cong S_{2,2}$.
\end{remark}

\section{Preliminaries}
We introduce some useful tools in this section.

Hermitian matrices have real eigenvalues. The Cauchy interlace theorem states that the eigenvalues
of a Hermitian matrix $ A $ of order $ n $ are interlaced with those of any principal submatrix.
\begin{lemma}[Cauchy Interlace Theorem]\label{Cauchy}
	Let $ A $ be a Hermitian matrix of order $ n $, and let $ B $ be a principal
	submatrix of $ A $ of order $ m $. If $\lambda_1(A)\ge\lambda_2(A)\ge\cdots\ge\lambda_n(A) $ lists the eigenvalues of $ A $ and $\mu_1(B)\ge \mu_2(B)\ge \cdots \ge \mu_m(B)$ the eigenvalues of $ B $, then
	$$ \lambda_{n-m+i}(A) \le \mu_i(B) \le \lambda_i(A)  \mbox{\ \ \ for\ } i = 1, \ldots, m. $$
\end{lemma}

Now we give a technical lemma which will be used to calculate distance eigenvalues. Let $M$ be a real $n\times n$ matrix, and let $\mathcal{N}=\{1,2,\ldots,n\}$. Given a partition $\Pi:\mathcal{N}=\mathcal{N}_1\cup \mathcal{N}_2\cup \cdots \cup \mathcal{N}_k$,  the matrix $M$ can be correspondingly partitioned as
$$
M=\left(\begin{array}{ccccccc}
M_{1,1}&M_{1,2}&\cdots &M_{1,k}\\
M_{2,1}&M_{2,2}&\cdots &M_{2,k}\\
\vdots& \vdots& \ddots& \vdots\\
M_{k,1}&M_{k,2}&\cdots &M_{k,k}\\
\end{array}\right).
$$
The \textit{quotient matrix} of $M$ with respect to $\Pi$ is defined as the $k\times k$ matrix $B_\Pi=(b_{i,j})$ where $b_{i,j}$ is the  average value of all row sums of  $M_{i,j}$.
The partition $\Pi$ is called \textit{equitable} if each block $M_{i,j}$ of $M$ has constant row sum $b_{i,j}$.
Also, we say that the quotient matrix $B_\Pi$ is \textit{equitable} if $\Pi$ is an  equitable partition of $M$.

\begin{lemma}{\rm(\cite{BH,GR})}\label{lem-eq}
	Let $M$ be a real symmetric matrix, and let $B_\Pi$ be an equitable quotient matrix of $M$. Then the eigenvalues of  $B_\Pi$ are also eigenvalues of $M$. Furthermore, if $M$ is nonnegative and irreducible, then $$\lambda_1(M)=\lambda_1(B_\Pi).$$
\end{lemma}

\section{Graphs with three distinct distance eigenvalues}
Koolen, Hayat and Iqbal \cite{Koolen} demonstrated that the graphs with exactly three distinct distance eigenvalues fall into four classes.
\begin{thm}[\cite{Koolen}]\label{3e}
	Let $ G $ be an $ n $-vertex connected graph with exactly three distinct distance
	eigenvalues $ \lambda_1(D(G)) > \lambda_2(D(G)) > \lambda_3(D(G)) $, with respective multiplicities $ m_1 = 1, m_2, m_3 $. Then one of the following holds.
	
	\noindent$ (i) $ $ G $ is complete bipartite;
	
	\noindent$ (ii) $ $ G $ is regular complete multipartite;	
	
	\noindent$ (iii) $ $ n $ is odd and $ \lambda_1(D(G)) = c(\frac{n-1}{2}) $ with $ 3 \le c \in \mathbb{Z}, $ and $ m_2 = m_3\ge 2 $;	
	
	\noindent$ (iv) $ $ \lambda_1(D(G)), \lambda_2(D(G)), \lambda_3(D(G)) \in \mathbb{Z}$, $ \lambda_2(D(G))\ge 0 $ and $ \lambda_3(D(G)) \le -3 $.
\end{thm}

Based on this useful result, we respectively characterize $ \{C_3,C_4\} $-free graphs with three distinct distance eigenvalues and $\lambda_3(D(G))=-3$ (see Problem \ref{p1}) and all trees with three distinct distance eigenvalues in this section. We firstly discuss Problem \ref{p1} in $ \{C_3,C_4\} $-free graphs.

\begin{lemma}[\cite{Koolen,lin}]\label{smallest}
	Let $ G $ be a connected graph. Then $ \lambda_n(D(G))\le -d $
	where $ d $ is the diameter of $ G $ and the equality holds if and only if $ G $ is a complete multipartite
	graph.
\end{lemma}

A Moore graph is a regular graph of diameter $ d $ and girth $ 2d + 1 $. Consider a $ k $-regular Moore graph $ G $ of diameter $d=2 $. The unique graphs for $ d=2 $ exist for $ k = 2, 3,7 $ and possibly for $ k = 57 $ (which is undecided), but for no other degree.



\begin{lemma}{\rm(\cite{57,86})}\label{ad}
	Let $ G $ be a $ k $-regular graph on $ n $ vertices with diameter at
	most $ 2 $ and adjacency spectrum $\lambda_1= k,\lambda_2,\lambda_3,\ldots,\lambda_n$. Then the distance spectrum of $ G $
	is $ 2n- 2 -k, -(2 + \lambda_2), -(2 + \lambda_3),\ldots , -(2 + \lambda_n) $.
\end{lemma}

\begin{lemma}{\rm(\cite{HS})}\label{222}
	A Moore graph with diameter $ d= 2 $ and degree $ k $ contains exactly three distinct adjacency eigenvalues $ k, (-1
	+\sqrt{4k-3})/2, (-1
	-\sqrt{4k-3})/2 $.
\end{lemma}

{\flushleft \textbf{Proof of Theorem \ref{t3}.}}
	Let $ G $ be a $ \{C_3,C_4\} $-free connected graph with diameter $d$ and three distinct distance eigenvalues
	$\lambda_1(D(G))>\lambda_2(D(G))>\lambda_3(D(G))$ such that $ \lambda_3(D(G))=-3 $. By Lemma \ref{smallest}, we get $ -3=\lambda_3(D(G))\le -d $, which implies $d\le 3$. Note that $ d\neq 3 $ since
	the equality condition of Lemma \ref{smallest}. So we have the diameter $d\le 2$. If $d=1$, then the graph $ G $ is complete and $ Spec(D(G))=\{n-1,[-1]^{n-1}\}  $, a contradiction. Thus, the diameter of $ G $ is 2. Besides, it is obvious that the girth of $ G $ is at most 5, which yields that the girth of $ G $ is 5 since $ G $ is $ \{C_3,C_4\} $-free. So $ G $ is a Moore graph with diameter $ d=2 $. Referring to Lemmas \ref{ad} and \ref{222}, we have $-3=-(2+\frac{-1+\sqrt{4k-3}}{2})$, implying $k=3$. Thus, $ G $ is the Petersen graph.

	If $ G $ is the Petersen graph, then $Spec(D(G))=\{15, [0]^4, [-3]^5\}$. The proof is complete. \qed

\noindent{\bf Remark 2.} By Theorem \ref{3e}, if $ G $ contais three distinct distance eigenvalues
$\lambda_1(D(G))>\lambda_2(D(G))>\lambda_3(D(G))$ such that $ \lambda_3(D(G))=-3 $, then $\lambda_1(D(G))$ and $\lambda_2(D(G))$ are integral.

\noindent{\bf Remark 3.}
Consider $C_3 $-free connected graphs in Problem \ref{p1}.  Let $ G $ be a $C_3 $-free connected graph with three distinct distance eigenvalues
$\lambda_1(D(G))>\lambda_2(D(G))>\lambda_3(D(G))$ such that $ \lambda_3(D(G))=-3  $. Referring to the proof of Theorem \ref{t3}, we only need to discuss one additional case that the girth of $ G $ is 4. Since the graph $ G $ is connected and $C_3 $-free, then $ G $ contains $ F_1 $ or $ F_2 $ (see Fig. \ref{ff}) as an induced subgraph. Let $
B_{F_1}=\left(\begin{array}{ccccccc}
0&1&2&1&1\\
1&0&1&2&2\\
2&1&0&1&2\\
1&2&1&0&2\\
1&2&2&2&0\\
\end{array}\right)
$ and $
B_{F_2}=\left(\begin{array}{ccccccc}
0&1&2&1&2\\
1&0&1&2&1\\
2&1&0&1&2\\
1&2&1&0&1\\
2&1&2&1&0\\
\end{array}\right).
$ Note that $ B_{F_1} $ or $ B_{F_2} $ is a principle submatrix of $ D(G) $. Then $\lambda_2(D(G))\ge \lambda_2(B_{F_1})$ or $\lambda_2(D(G))\ge \lambda_2(B_{F_2})$ by Lemma \ref{Cauchy}. We obtain $  \lambda_2(B_{F_1})=0.0841 $ and $\lambda_2(B_{F_2})=0.3542 $ by a simple calculation. Hence, $\lambda_2(D(G))\ge 1$ as $\lambda_2(D(G))$ is an integer.
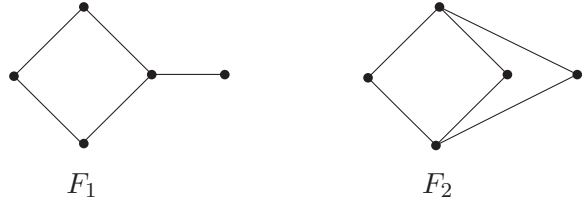
\begin{figure}[htp]
	\setlength{\unitlength}{0.9pt}
	\begin{center}
		\begin{picture}(234.2,69.6)
		\put(29.0,69.6){\circle*{4}}
		\put(0.0,40.6){\circle*{4}}
		\qbezier(29.0,69.6)(14.5,55.1)(0.0,40.6)
		\put(29,12.3){\circle*{4}}
		\qbezier(0.0,40.6)(14.1,26.5)(28.3,12.3)
		\put(57.3,41.3){\circle*{4}}
		\qbezier(29.0,69.6)(43.1,55.5)(57.3,41.3)
		\qbezier(57.3,41.3)(42.8,26.8)(28.3,12.3)
		\put(87.7,41.3){\circle*{4}}
		\qbezier(57.3,41.3)(72.5,41.3)(87.7,41.3)
		\put(176.9,69.6){\circle*{4}}
		\put(147.2,39.9){\circle*{4}}
		\qbezier(176.9,69.6)(162.0,54.7)(147.2,39.9)
		\put(175.5,11.6){\circle*{4}}
		\qbezier(147.2,39.9)(161.3,25.7)(175.5,11.6)
		\put(205.2,41.3){\circle*{4}}
		\qbezier(176.9,69.6)(191.0,55.5)(205.2,41.3)
		\qbezier(205.2,41.3)(190.3,26.5)(175.5,11.6)
		\put(234.2,41.3){\circle*{4}}
		\qbezier(176.9,69.6)(205.5,55.5)(234.2,41.3)
		\qbezier(234.2,41.3)(204.8,26.5)(175.5,11.6)
		\put(21.8,0.0){\makebox(0,0)[tl]{$F_1$}}
		\put(169.7,0.0){\makebox(0,0)[tl]{$F_2$}}
		\end{picture}
		\caption{The structure of induced subgraphs $F_1$ and $F_2$.}\label{ff}
	\end{center}
\end{figure}


It is the position to characterize all trees with three distinct distance eigenvalues. Before the proof, we show some useful lemmas.
The \textit{Laplacain matrix} of a graph $G$ is denoted by $L(G)=\Delta(G)-A(G)$, where $\Delta(G)=diag(d_1,\ldots,d_n)$, and $A(G)$ is the adjacency matrix of $G$.
In Theorem \ref{mer}, Merris obtained a relation between the distance eigenvalues and Laplacain eigenvalues of a tree. This relation plays an important role in the proof of Theorem \ref{3tree}.
\begin{lemma}[\cite{BH}]\label{dd1}
Let $G$ be a connected graph with diameter $ d $. Then $G$ has at least
$ d + 1 $ distinct Laplacain eigenvalues.
	\end{lemma}
An immediate consequence of Theorem \ref{dett} is as follows:
\begin{lemma}[\cite{GP}]\label{26}
	 Let $ T $ be a tree with n vertices where $ n \ge 2 $. Then $ D(T) $ has $ 1 $ positive and $ n-1 $ negative eigenvalues.
\end{lemma}

{\flushleft \textbf{Proof of Theorem \ref{3tree}.}}
	If $ T $ is a star with $ n $ vertices, then we have $Spec(D(T))=\{n-2+\sqrt{n^2-3n+3},n-2-\sqrt{n^2-3n+3},[-2]^{n-2}\}$ by a simple calculation and the condition is obviously sufficient.

Next suppose $ T$ is an $ n $-vertex tree with exactly three distinct distance
eigenvalues $ \lambda_1(D(T)) > \lambda_2(D(T)) > \lambda_3(D(T)) $, with respective multiplicities $ m_1 = 1, m_2, m_3 $. By Theorem \ref{3e}, $ T $ is a star if $ (i) $ holds. Thus, we assume that the tree $ T $ is not a star which implies that $ T $ is neither complete bipartite nor regular complete multipartite. Besides, we notice that $\lambda_1(D(T)) >0>\lambda_2(D(T)) > \lambda_3(D(T)) $ by Lemma \ref{26}. According to the above statements, we rule out the possibility of $ (i), (ii) $ and $ (iv) $ in Theorem \ref{3e}. So we only need to consider $ (iii) $ in Theorem \ref{3e} in the following argument.

Let $ n=2q+1 $ where $q\ge 2$ since $ T \cong P_3$ is a star when $q=1$. Then we have $ \lambda_1(D(T)) = cq $ with $ 3 \le c \in \mathbb{Z}$, and $ m_2 = m_3=q $. Note that $det(D(T))=\lambda_1(D(T))\lambda_2(D(T))^{m_2}\lambda_3(D(T))^{m_3}=(-1)^{n-1}(n-1)2^{n-2} $ by Theorem \ref{dett}. So we obtain
\begin{eqnarray*}
	(-1)^{2q}\cdot2q\cdot2^{2q-1}=cq\cdot\lambda_2(D(T))^q\cdot\lambda_3(D(T))^q,
\end{eqnarray*}
which implies
\begin{eqnarray}\label{qq1}
4^q=c(\lambda_2(D(T))\lambda_3(D(T)))^q.
\end{eqnarray}
It is clear that $4\mid\lambda_2(D(T))\lambda_3(D(T))$ and $ \lambda_2(D(T))\lambda_3(D(T)) $ equals to 1 or 2, otherwise $\lambda_2(D(T))\lambda_3(D(T))=4$ and $c=1$, contradicting to $ 3 \le c \in \mathbb{Z}$.
Since the trace of the distance matrix is zero, the sum of distance eigenvalues is zero, and hence we obtain
\begin{eqnarray}\label{qq2}
\lambda_1(D(T))+q\lambda_2(D(T))+q\lambda_3(D(T))=0.
\end{eqnarray}
If $ \lambda_2(D(T))\lambda_3(D(T))=1 $, then we get $c=4^q$, $\lambda_1(D(T))=q\cdot4^q$. Therefore, we can multiply (\ref{qq1}) and (\ref{qq2}), obtaining
\begin{eqnarray}\label{66}
 \lambda_2(D(T))=\frac{-4^q+\sqrt{16^q-4}}{2}\  \mbox{and}\  \lambda_3(D(T))=\frac{-4^q-\sqrt{16^q-4}}{2}.
\end{eqnarray}
If $ \lambda_2(D(T))\lambda_3(D(T))=2 $, then we have $c=2^q$, $\lambda_1(D(T))=q\cdot2^q$. Rearranging equations  (\ref{qq1}) and (\ref{qq2}), we get
\begin{eqnarray}\label{77}
\lambda_2(D(T))=\frac{-2^q+\sqrt{4^q-8}}{2}
\  \mbox{and}\
\lambda_3(D(T))=\frac{-2^q-\sqrt{4^q-8}}{2}
\end{eqnarray}
Armed with the values of $ \lambda_2(D(T)) $ and $ \lambda_3(D(T)) $, it is obvious that $-2\notin Spec(D(T))$.

By Theorem \ref{mer}, $ T$ has at most five distinct Laplacain eigenvalues and the diameter of $ T$ is at most four in view of Lemma \ref{dd1}. Hence, the analysis is partitioned in the following three cases according to the diameter $ d $ of the tree $T$. Firstly suppose $d\le 2$, then $T$ is a star, a contradiction. Next we consider $d= 3$. Note that $-2\notin Spec(D(T))$ and $ T\in  \mathbb{T}_3 =\{S_{s+1,t+1}|\ s\in \mathbb{N}\ \mbox{and}\ t\in \mathbb{N}\} $. Then we conclude that $ T \cong S_{1,1}\cong P_4$ referring to Lemma \ref{-2} and $Spec(D(P_4))=\{2 + \sqrt{10},  \sqrt{2}-2,2-\sqrt{10} ,-2- \sqrt{2}\}$, contradicting the hypothesis. Finally, we suppose $d=4$. All trees with diameter 4 are shown in Fig. \ref{tree4}.
Since $-2\notin Spec(D(T))$, together with Lemma \ref{-2}, implies that $ T $ is either $ T_4^1 $ ($ s=s'=0, p=k_1=\cdots=k_t=1 $) or $ T_4^2 $ ($ s=s'=p=0, k_1=\cdots=k_t=1 $). Considering the order of $ T $ is odd, so $ T\cong T_4^2 $ and from now on we shall calculate the distance spectrum of $ T_4^2$.
\begin{figure}[htp]
	\setlength{\unitlength}{0.9pt}
	\begin{center}
		\begin{picture}(253.0,155.9)
		\put(17.4,77.6){\circle*{4}}
		\put(132.0,77.6){\circle*{4}}
		\qbezier(17.4,77.6)(45.7,77.6)(132.0,77.6)
		\put(45.7,77.6){\circle*{4}}
		\put(248.7,77.6){\circle*{4}}
		\qbezier(132.0,77.6)(220.4,77.6)(248.7,77.6)
		\put(220.4,77.6){\circle*{4}}
		\put(19.6,88.5){\circle*{4}}
		\qbezier(45.7,77.6)(32.6,83.0)(19.6,88.5)
		\put(45.7,105.9){\circle*{4}}
		\qbezier(45.7,77.6)(45.7,91.7)(45.7,105.9)
		\put(245.8,87.7){\circle*{4}}
		\qbezier(220.4,77.6)(233.1,82.7)(245.8,87.7)
		\put(220.4,105.9){\circle*{4}}
		\qbezier(220.4,77.6)(220.4,91.7)(220.4,105.9)
		\put(118.2,139.2){\circle*{4}}
		\qbezier(132.0,77.6)(125.4,106.6)(118.2,139.2)
		\put(125.4,106.6){\circle*{4}}
		\put(148.6,139.2){\circle*{4}}
		\qbezier(132.0,77.6)(140.7,106.6)(148.6,139.2)
		\put(140.7,106.6){\circle*{4}}
		\put(116.0,38.4){\circle*{4}}
		\qbezier(132.0,77.6)(124.0,58.0)(116.0,38.4)
		\put(150.1,38.4){\circle*{4}}
		\qbezier(132.0,77.6)(141.0,58.0)(150.1,38.4)
		\put(129.8,122.5){\circle*{2}}
		\put(137.0,122.5){\circle*{2}}
		\put(133.4,122.5){\circle*{2}}
		\put(129.1,55.1){\circle*{2}}
		\put(136.3,55.1){\circle*{2}}
		\put(132.7,55.1){\circle*{2}}
		\put(227.7,95.0){\circle*{2}}
		\put(234.2,90.6){\circle*{2}}
		\put(231.3,92.8){\circle*{2}}
		\put(38.4,96.4){\circle*{2}}
		\put(33.4,91.4){\circle*{2}}
		\put(35.5,93.5){\circle*{2}}
		\put(115.4,10.0){\makebox(0,0)[tl]{the set $\mathbb{T}_4$}}
		\put(92.8,115.3){\circle*{4}}
		\qbezier(125.4,106.6)(109.1,110.9)(92.8,115.3)
		\put(172.6,115.3){\circle*{4}}
		\qbezier(140.7,106.6)(156.6,110.9)(172.6,115.3)
		\put(116.0,122.5){\circle*{2}}
		\put(110.2,116.7){\circle*{2}}
		\put(113.1,119.6){\circle*{2}}
		\put(151.5,124.0){\circle*{2}}
		\put(157.3,118.2){\circle*{2}}
		\put(154.4,121.1){\circle*{2}}
		\put(30.0,101.3){\makebox(0,0)[tl]{$s$}}
		\put(232.1,103.6){\makebox(0,0)[tl]{$s'$}}
		\put(130.4,135.8){\makebox(0,0)[tl]{$t$}}
		\put(100.3,134.1){\makebox(0,0)[tl]{${k_1}$}}
		\put(160.1,134.9){\makebox(0,0)[tl]{${k_t}$}}
		\put(132.1,45.5){\makebox(0,0)[tl]{$p$}}
		\end{picture}
		\caption{The trees with diameter 4.}\label{tree4}
	\end{center}
\end{figure}

\begin{figure}[htp]
	\setlength{\unitlength}{0.9pt}
	\begin{center}
		\begin{picture}(290.0,76.1)
		\put(0.0,46.4){\circle*{4}}
		\put(58.0,46.4){\circle*{4}}
		\qbezier(0.0,46.4)(29.0,46.4)(58.0,46.4)
		\put(29.0,46.4){\circle*{4}}
		\put(116.7,46.4){\circle*{4}}
		\qbezier(58.0,46.4)(87.0,46.4)(116.7,46.4)
		\put(87.0,46.4){\circle*{4}}
		\put(58.0,76.1){\circle*{4}}
		\qbezier(58.0,76.1)(58.0,61.3)(58.0,46.4)
		\put(29.0,3.6){\circle*{4}}
		\qbezier(58.0,46.4)(43.5,25.4)(29.0,3.6)
		\put(43.5,25.4){\circle*{4}}
		\put(88.5,2.9){\circle*{4}}
		\qbezier(58.0,46.4)(73.2,24.7)(88.5,2.9)
		\put(73.2,24.7){\circle*{4}}
		\put(55.1,13.1){\circle*{2}}
		\put(62.4,13.1){\circle*{2}}
		\put(58.7,13.1){\circle*{2}}
		\put(174.0,47.1){\circle*{4}}
		\put(232.0,47.1){\circle*{4}}
		\qbezier(174.0,47.1)(203.0,47.1)(232.0,47.1)
		\put(203.0,47.1){\circle*{4}}
		\put(290.0,47.1){\circle*{4}}
		\qbezier(232.0,47.1)(261.0,47.1)(290.0,47.1)
		\put(261.0,47.1){\circle*{4}}
		\put(204.5,4.4){\circle*{4}}
		\qbezier(232.0,47.1)(218.2,24.7)(204.5,4.4)
		\put(218.2,24.7){\circle*{4}}
		\put(261.0,4.4){\circle*{4}}
		\qbezier(232.0,47.1)(246.5,26.1)(261.0,4.4)
		\put(246.5,26.1){\circle*{4}}
		\put(229.1,13.1){\circle*{2}}
		\put(236.4,13.1){\circle*{2}}
		\put(232.7,13.1){\circle*{2}}
		\put(50.0,0.0){\makebox(0,0)[tl]{$T_4^1$}}
		\put(224.8,1.5){\makebox(0,0)[tl]{$T_4^2$}}
		\end{picture}
		\caption{The trees $T_4^2$ and $T_4^2$.}\label{tree42}
	\end{center}
\end{figure}
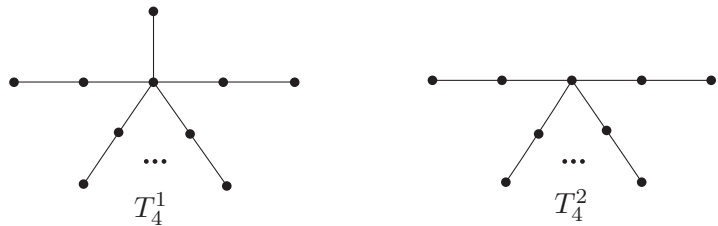
Let $V_i=\{v\in V( T_4^2)|\ deg(v)=i\}$. Clearly, $|V_1|=q, |V_2|=q , |V_q|=1$ and $ V( T_4^2) =V_1\cup V_2\cup V_q $ is an equitable partition of $ D(T_4^2) $. So the corresponding quotient matrix is
	
\begin{eqnarray*}
	B(T_4^2)=\left[
	\begin{array}{ccc}
		0 &q &2q \\ [2mm]
     	1 &2(q-1)&3(q-1)+1  \\ [2mm]
		2 &3(q-1)+1 &4(q-1)
	\end{array}
	\right],
\end{eqnarray*}
the characteristic polynomial of the matrix $ B(T_4^2) $ is
\begin{eqnarray*}
g(\lambda)=\lambda^3 - (6q - 6)\lambda^2 - (q^2 + 9q - 4)\lambda - 4q.
\end{eqnarray*}
Since the partition is equitable, then by Lemma \ref{lem-eq}, the characteristic polynomial $g(\lambda)$ can be expressed as one of the following three forms:

\noindent\textbf{Form 1.} $g(\lambda)=(\lambda-\lambda_1(D(T)))(\lambda-\lambda_2(D(T)))(\lambda-\lambda_3(D(T)))$.

We obtain that $\lambda_1(D(T))\lambda_2(D(T))\lambda_3(D(T))=4q$, which implies that $4q=q\cdot4^q$ or $4q=2q\cdot2^q$ when $q\ge 2$, a contradiction.

\noindent\textbf{Form 2.} $g(\lambda)=(\lambda-\lambda_1(D(T)))(\lambda-\lambda_2(D(T)))^2$.

Then we get that $ h_1(q)=\lambda_1(D(T))+2\lambda_2(D(T))=6q-6$, which yields that $\lambda_1(D(T))+2\lambda_2(D(T))-6q+6 =0$.
Substituting (\ref{66}) and $q\ge 2$, we have
\begin{eqnarray*}
h_1(q)&=&q\cdot 4^q-4^q+\sqrt{16^q-4}-6q+6\\
      &>&2\cdot 4^q-4^q-6q+6\\
      &=&4^q-6q+6\\
      &\ge& 4^2-6\cdot 2+6=10>0.
\end{eqnarray*}
This contradicts the equation $ \lambda_1(D(T))+2\lambda_2(D(T))-6q+6 =0$. Next, we consider (\ref{77}) and get $h_2(q)= q\cdot 2^q-2^q+\sqrt{4^q-8}-6q+6=0$. By a simple calculation, we get $ h_2(2)=2\sqrt{2}-2\neq0 $. So we just need to consider $q\ge 3$. Then we have
\begin{eqnarray*}
	h_2(q)&=&q\cdot 2^q-2^q+\sqrt{4^q-8}-6q+6\\
	&>&3\cdot 2^q-2^q-6q+6\\
	&=&2\cdot 2^q-6q+6\\
	&\ge& 2\cdot 2^3-6\cdot 3+6=4>0.
\end{eqnarray*}
This also contradicts the equation $ \lambda_1(D(T))+2\lambda_2(D(T))-6q+6 =0$.

\noindent\textbf{Form 3.} $g(\lambda)=(\lambda-\lambda_1(D(T)))(\lambda-\lambda_3(D(T)))^2$.

The specific steps are the same as above. We need to prove that there exist integers $q\ge 2$ to make
the equation $ \lambda_1(D(T))+2\lambda_3(D(T))-6q+6 =0$ hold, that is,
$l_1(q)=q\cdot 4^q-4^q-\sqrt{16^q-4}-6q+6=0$ or $l_2(q)=q\cdot 2^q-2^q-\sqrt{4^q-8}-6q+6=0$.
Note that $l_1(2)=10-6\sqrt{7}\neq 0$, $l_2(2)=-2-2\sqrt{2}\neq 0$ and $l_2(3)=4-2\sqrt{14}\neq 0$. By calculations, we have
\begin{eqnarray*}
	l_1(q)&=&q\cdot 4^q-4^q-\sqrt{16^q-4}-6q+6\\
	&>&3\cdot 4^q-4^q-4^q-6q+6\\
	&=&4^q-6q+6>0\ \ \mbox{when $q \ge 3$ }
\end{eqnarray*}
and
\begin{eqnarray*}
	l_2(q)&=&q\cdot 2^q-2^q-\sqrt{4^q-8}-6q+6\\
	&>&4\cdot 2^q-2^q-2^q-6q+6\\
	&=&2\cdot 2^q-6q+6>0\ \ \mbox{when $q \ge 4$ }.
\end{eqnarray*}
The results contradict the equation $ \lambda_1(D(T))+2\lambda_3(D(T))-6q+6 =0$.

Taken together, the tree cannot contain exactly three distinct distance eigenvalues if it is not a star. The proof is complete. \qed

\section{Trees with distance eigenvalue equal to $ -1 $}
At the beginning of this section, we characterize a series of trees containing $-1$ as their distance eigenvalues. Now we shall prove that any tree in $ \mathcal{T}(4n+2) $ contains $ -1 $ as its distance eigenvalue.

{\flushleft \textbf{Proof of Theorem \ref{-1}.}}
We asset that all trees in $ \mathcal{T}(4n+2) $ have $ -1 $ as their distance eigenvalues and the sum of components of each corresponding
eigenvector is 0.  We apply induction on $ n $ to prove this assertion.
Clearly, $ \mathcal{T}(2)=\{P_2\} $ and $ \mathcal{T}(6)=\{P_6\} $ and the assertion is trivial.
So suppose $ n=k+1 $ and the assertion holds for $ n\le k $.
For any tree $ T_k\in \mathcal{T}(4k+2) $, we have a vector $\mathbf{x}_k=(x_1^{(k)},x_2^{(k)},\ldots, x_{4k+1}^{(k)},x_{4k+2}^{(k)})^T $ to satisfy $D(T_k)\mathbf{x}_k=-\mathbf{x}_k$ and $\sum_{j=1}^{4k+2}x_j^{(k)}=0$. Let $ P_4=w_1w_2w_3w_4 $ be a path on four vertices. We assume that the tree $ T_{k+1} $ is obtained by attaching an edge between the endvertex $ w_1 $ in $ P_4$ and a vertex in $ T_k $, called $ v_p $ without loss of generality. It is clear that $ T_{k+1} $ belongs to $ \mathcal{T}(4k+6)$.
Let $ \mathbf{x}_{k+1}=(\mathbf{x}_{k}^T,y_1,y_2,y_3,y_4)^T=(x_1^{(k)},x_2^{(k)},\ldots, x_{4k+1}^{(k)},x_{4k+2}^{(k)},x_{p}^{(k)},-x_{p}^{(k)},-x_{p}^{(k)},x_{p}^{(k)})^T$ where $ y_i $ is corresponding to the vertex $ w_i $ of the pendent path $ P_4 $.
To prove that $ -1 $ is an eigenvalue of $ D(T_{k+1}) $, it suffices to show that $D(T_{k+1})\mathbf{x}_{k+1}=-\mathbf{x}_{k+1}$.
By a simple calculation, we obtain
\begin{eqnarray*}
	D(T_{k+1})=\left[
	\begin{array}{ccc|cccc}
		\ &\ &\ &\ &\ &\ &\ \\ 
		\ &{\huge D(T_k)}&\ &\ &\ &{\huge B}&\ \\
		\ &\ &\ &\ &\ &\ &\ \\  \hline
		\ &\ &\ & 0&1&2&3 \\
		\ &\ &\ & 1&0&1&2 \\
		\ &{\huge B^T} &\ &2&1&0&1 \\
		\ &\ &\ &3&2&1&0
	\end{array}
	\right]
\end{eqnarray*}
where
\begin{eqnarray*}
	B=\left[
	\begin{array}{cccc}
		d_{1,p}+1&d_{1,p}+2&d_{1,p}+3&d_{1,p}+4 \\ 
		d_{2,p}+1&d_{2,p}+2&d_{2,p}+3&d_{2,p}+4 \\ 
		\vdots&\vdots&\vdots&\vdots \\
		d_{4k+1,p}+1&d_{4k+1,p}+2&d_{4k+1,p}+3&d_{4k+1,p}+4 \\ 
		d_{4k+2,p}+1&d_{4k+2,p}+2&d_{4k+2,p}+3&d_{4k+2,p}+4 \\
		
	\end{array}
	\right].
\end{eqnarray*}
This impiles that
\begin{eqnarray*}
	-x_i^{(k)}&=&-x_i^{(k)}+0\\
	&=&\sum_{j=1}^{4k+2}d_{i,j}x_j^{(k)}+(d_{i,p}+1)y_1+(d_{i,p}+2)y_2+(d_{i,p}+3)y_3+(d_{i,p}+4)y_4
\end{eqnarray*}
for $ i=1,2,\ldots,4k+1,4k+2 $.
Besides, there always holds
\begin{eqnarray*}
	-y_1 &=&-x_p^{(k)}=(-x_p^{(k)}+0)-x_p^{(k)}-2x_p^{(k)}+3x_p^{(k)}\\  &=&\sum_{j=1}^{4k+2}(d_{j,p}+1)x_j^{(k)}+y_2+2y_3+3y_4,
\end{eqnarray*}
and
\begin{eqnarray*}
	-y_2 &=&x_p^{(k)}=(-x_p^{(k)}+2\times 0)+x_p^{(k)}-x_p^{(k)}+2x_p^{(k)}\\  &=&\sum_{j=1}^{4k+2}(d_{j,p}+2)x_j^{(k)}+y_1+y_3+2y_4,
\end{eqnarray*}
and
\begin{eqnarray*}
	-y_3 &=&x_p^{(k)}=(-x_p^{(k)}+3\times 0)+2x_p^{(k)}-x_p^{(k)}+x_p^{(k)}\\  &=&\sum_{j=1}^{4k+2}(d_{j,p}+3)x_j^{(k)}+2y_1+y_2+y_4,
\end{eqnarray*}
and
\begin{eqnarray*}
	-y_4 &=&x_p^{(k)}=(-x_p^{(k)}+4\times 0)+3x_p^{(k)}-2x_p^{(k)}-x_p^{(k)}\\  &=& \sum_{j=1}^{4k+2}(d_{j,p}+4)x_j^{(k)}+3y_1+2y_2+y_3.\end{eqnarray*}
Thus we conclude that $ \mathbf{x}_{k+1}= (x_1^{(k)},x_2^{(k)},\ldots, x_{4k+1}^{(k)},x_{4k+2}^{(k)},x_p^{(k)},-x_p^{(k)},-x_p^{(k)},x_p^{(k)})^T $ is the corresponding eigenvector of the eigenvalue $ -1 $ of $ D(T_{k+1}) $ and the sum of components of $\mathbf{x}_{k+1}$ is 0 as well. This proof is complete because of the arbitrary choices of the tree $ T_k\in \mathcal{T}(4k+2) $ and the joining vertex $ v_p $ of $ T_k $. \qed


Considering the multiplicity of an eigenvalue, we pose the following problem naturally: For any $n\in \mathbb{N}$ and tree $ T\in \mathcal{T}(4n+2) $, is it right that the multiplicity of $ -1 $ as an eigenvalue of $ D(T) $ is one?
We start with the more easily addressable issues and give a positive answer to the above problem for $ P_{4n+2} $.
\begin{lemma}\label{path}{\rm(\cite{RP})}
	Let $ P_n $ be a path on $ n>2 $ vertices. The largest eigenvalue $ d_1 $ and other eigenvalues $ d_i $ of its distance matrix $ D(P_n) $ are as follows.
	
	\noindent$ (1)$  $ d_1=1/(\cosh\theta -1) $, where $ \theta $ is the positive solution of $\tanh(\theta/2)\tanh(n\theta/2)=1/n$.
	
	\noindent$ (2) $ $ d_i = 1/(\cos\theta -1) $, where $ \textbf{(I)} $ $ \theta $ is one of the $  [(n- 1)/2] $ solutions of $ \tan(\theta/2) \tan(n\theta/2) =-1 /n $ in the interval $ (0,\pi) $, or $ \textbf{(II)} $  $ \theta= (2m - 1)\pi/n $ for $ m = 1,\ldots,[n/2] $.
	
\end{lemma}

\begin{thm}
	Let $ P_k $ be a path on $ k\ge 2 $ vertices. The value $ -1 $ is an eigenvalue of $ D(P_k) $ if and only if $ k=4n+2 $ where $ n\in \mathbb{N}$. Moreover, the multiplicity of $ -1 $ as an eigenvalue of its distance matrix is one.
	
\end{thm}
\begin{proof}
	It follows from Theorem \ref{-1} that $ P_{4n+2}\in \mathcal{T}(4n+2) $ and $ -1 $ is an eigenvalue of $ D(P_{4n+2}) $.
	Conversely, suppose $ -1 $ is an eigenvalue of $ D(P_k) $. Then we have $ -1 = 1/(\cos\theta -1) $ by Lemma \ref{path} and it implies that $\theta=\pi/2 $ if $ \theta $ is in the interval $ (0,\pi) $. Considering $ \textbf{(I)} $ and $ \textbf{(II)} $ in Lemma \ref{path} (2), it is obvious that $ \tan(\pi/4) \tan(n\pi/4) \neq -1 /k  $ and $\theta=\pi/2 $ when $k=4m-2$. We conclude that $ -1 $ is an eigenvalue of $ D(P_k) $ when $k=4n+2$ for $ n\in \mathbb{N}$.
	
	We show that the multiplicity of $ -1 $ as an eigenvalue of $D(P_{4n+2}) $ is one in the following.	
	The assertion is obviously right for $ n=0 $ so we only consider the case for $n\ge1$. Let $ \mathbf{x}=(x_1,x_2,\ldots, x_{4n+1},x_{4n+2})^T\neq 0$ be an arbitrary eigenvector of the eigenvalue $ -1 $ of $D(P_{4n+2}) $.
	In order to prove this theorem, we need to show that each component of vector $\mathbf{x}$ is not zero, that is, $ x_i\neq 0 $ for any integer $ i\in [1,{4n+2}] $. By a sipmle calculation on $ D(P_{4n+2})\mathbf{x}=-\mathbf{x} $, we have
	\begin{eqnarray}\label{eq1}
	-x_q=\sum_{j=1}^{4n+2}|q-j|x_j
	\end{eqnarray}
	for $q=1,2,\ldots,4n+1,4n+2$. Then we have the following equations by (\ref{eq1}): 	
	\begin{eqnarray*}
		-(x_{2n+1}+x_{2n+2})=[\sum_{q=1}^{2n}(4n+3-2q)(x_q+x_{4n+3-q})]+(x_{2n+1}+x_{2n+2})
	\end{eqnarray*}	
	and
	\begin{eqnarray*}
		-(x_{2n}+x_{2n+3})&=&[\sum_{q=1}^{2n}(4n+3-2q)(x_q+x_{4n+3-q})]+3(x_{2n+1}+x_{2n+2})\\
		&=&-2(x_{2n+1}+x_{2n+2})+3(x_{2n+1}+x_{2n+2})\\
		&=&x_{2n+1}+x_{2n+2},
	\end{eqnarray*}
	and
	\begin{eqnarray*}
		-(x_{2n-1}+x_{2n+4})&=&[\sum_{q=1}^{2n-1}(4n+3-2q)(x_q+x_{4n+3-q})]+5[(x_{2n}+x_{2n+3})+(x_{2n+1}+x_{2n+2})]\\
		&=&-4(x_{2n}+x_{2n+3})-3(x_{2n+1}+x_{2n+2})+5[(x_{2n}+x_{2n+3})+(x_{2n+1}+x_{2n+2})]\\
		&=&(x_{2n}+x_{2n+3})+2(x_{2n+1}+x_{2n+2})\\
		&=&x_{2n+1}+x_{2n+2}.
	\end{eqnarray*}
	By parity of reasoning, it is clear that  $-(x_q+x_{4n+3-q})=x_{2n+1}+x_{2n+2}$, which implies that $$ -2(x_{2n+1}+x_{2n+2})=-\left[\sum_{q=1}^{2n}(4n+3-2q)(x_{2n+1}+x_{2n+2})\right]=-(4n^2+4n)(x_{2n+1}+x_{2n+2}).$$
	Hence
	\begin{eqnarray}\label{eq2}
	x_q+x_{4n+3-q}=0, \  \mbox{for} \ q=1,2,\ldots,2n,2n+1.
	\end{eqnarray}
	Furthermore, we also obtain the following equations by (\ref{eq1}):
	\begin{eqnarray*}
		-x_{2n+1}+x_{2n+2}=-\sum_{q=1}^{2n}x_q-x_{2n+1}+x_{2n+2}+\sum_{q=2n+3}^{4n+2}x_{q} \Rightarrow \sum_{q=1}^{2n}x_q=\sum_{q=2n+3}^{4n+2}x_{q},
	\end{eqnarray*}	
	
	\begin{eqnarray*}
		-x_{2n}+x_{2n+3}=-3\sum_{q=1}^{2n}x_q-x_{2n+1}+x_{2n+2}+3\sum_{q=2n+3}^{4n+2}x_{q} =-x_{2n+1}+x_{2n+2},
	\end{eqnarray*}	
	
	\begin{eqnarray*}
		-x_{2n-1}+x_{2n+4}&=&-5\sum_{q=1}^{2n-1}x_q-3x_{2n}-x_{2n+1}+x_{2n+2}+3x_{2n+3}+5\sum_{q=2n+4}^{4n+2}x_{q}\\ &=&5(x_{2n}-x_{2n+3})-4(x_{2n+1}-x_{2n+2})=x_{2n+1}-x_{2n+2},
	\end{eqnarray*}	
	and
	\begin{eqnarray*}
		-x_{2n-2}+x_{2n+5}&=&-7\sum_{q=1}^{2n-2}x_q-5x_{2n-1}-3x_{2n}-x_{2n+1}+x_{2n+2}+3x_{2n+3}+5x_{2n+4}+7\sum_{q=2n+5}^{4n+2}x_{q}\\
		&=&-7(-x_{2n}-x_{2n-1}+x_{2n+3}+x_{2n+4})+x_{2n+1}-x_{2n+2}\\
		&=&x_{2n+1}-x_{2n+2}.
	\end{eqnarray*}
	By parity of reasoning, it is clear that for $ q=1,2,\ldots,2n,2n+1 $
	
	\begin{align}\label{eq3}
	\begin{split}
	\left \{
	\begin{array}{ll}
	-x_q+x_{4n+3-q}=-x_{2n+1}+x_{2n+2}    &\mbox{when\ } 2n+1-q\equiv 0,1 (\mbox{mod} 4), \\[3mm]
	-x_q+x_{4n+3-q}=x_{2n+1}-x_{2n+2}     &\mbox{when\ } 2n+1-q\equiv 2,3 (\mbox{mod} 4).
	\end{array}
	\right.
	\end{split}
	\end{align}
	(\ref{eq2}) and (\ref{eq3}) give $|x_i|=|x_j|$ for any $i,j=1,\dots {4n+2}$. Since the vector $ \mathbf{x}$ is nonzero, we have $|x_i|=a\in \mathbb{R}\backslash\{0\}$ for any integer $i\in [1,{4n+2}]$.

\end{proof}

Does the set $\mathcal{T}(4n+2)$ cover all the trees whose distance spectra contain $ -1 $? The answer is negative. In the following, we characterize all connected trees of diameter at most three with distance eigenvalues that are equal to $ -1 $. Besides, we show a special tree whose distance spectrum contains $ -1 $ but does not belong to $ \mathcal{T}(4n+2) $.

A \textit{pendent vertex} of a tree is a vertex of degree one.
A \textit{pendent
	neighbor} is a vertex adjacent to a pendent vertex.

\begin{lemma}{\rm(\cite{Collins})}\label{-2}
Let $ T $ be a tree with $ n_1 $ pendent vertices and $ n_1' $ pendent
neighbors. Then among the distance eigenvalues of $ T $,  -2 as an eigenvalue occurs with multiplicity
at least $ n_1-n_1' $.
\end{lemma}

{\flushleft \textbf{Proof of Remark \ref{sp}.}}
We can easily get by computer that $ Spec(D(P_2))=\{1,-1\} $
and $ Spec(D(S_{2,2}))=\{10, -1, -\frac{5}{2} - \frac{\sqrt{17}}{2}, -\frac{5}{2}+ \frac{\sqrt{17}}{2},[-2]^2\} $.

Denote by $ \mathbb{T}_d $ the set of trees of diameter $ d $. We distinguish the following two cases:

\noindent{\bf Case 1.} $ d\le 2. $

The result is trivial for the case $ d=1 $ and $ Spec(D(K_2))=\{1,-1\} $. For $d=2$, the graphs in $ \mathbb{T}_2 $ are stars and $ \mathbb{T}_2 =\{K_{1,t+2}|\ t\in \mathbb{N}\}$. Let $J_{t+2}$ be the matrix of order $t+2 $ with all entries equal to 1. Then the distance matrix of $ K_{1,t+2} $ is
\begin{eqnarray*}
	D=\left[
	\begin{array}{ccc|c}
		\ &\ &\ &1\\ 
		\ &{\huge 2(J-I)_{t+2}}&\ &\vdots  \\
		\ &\ &\ &1 \\  \hline
		1 &\cdots &1 & 0
	\end{array}
	\right]
\end{eqnarray*}
and the distance characteristic polynomial of $ K_{1,t+2} $ is

\begin{eqnarray*}
	det(\lambda I-D)&=&(\lambda+2)^{t+1}\left|
	\begin{array}{ccc}
		\lambda-(2t+2)&\ &-1 \\
		\ &\ &\  \\
		-(t+2)&\ &\lambda
	\end{array}
	\right|\\[2mm]
	&=&(\lambda+2)^{t+1}\left[\lambda^2-(2t+2)\lambda-(t+2)\right].
\end{eqnarray*}
If $-1$ is an eigenvalue of $D(K_{1,t+2})$, then $(-1)^2-(2t+2)(-1)-(t+2)=0$, which implies that $t=-1 $, contradicting with $ t\ge 0 $.	
Thus, only graph $P_2$ contains $-1$ as its distance  eigenvalue among all trees with diameter $ d\le 2 $.

\noindent{\bf Case 2.} $ d=3. $

Note that $ \mathbb{T}_3 =\{S_{s+1,t+1}|\ s\in \mathbb{N}\ \mbox{and}\ t\in \mathbb{N}\}$. For any tree $ T\in \mathbb{T}_3 $, there is an equitable partition of $ D(T) $ and the corresponding quotient matrix is

\begin{eqnarray*}
	B(T)=\left[
	\begin{array}{cccc}
		2s &1 &2 &3(t+1)\\ [2mm]
		s+1 &0&1 &2(t+1) \\ [2mm]
		2(s+1) &1 &0 &t+1 \\ [2mm]
		3(s+1) &2 &1 & 2t
	\end{array}
	\right],
\end{eqnarray*}
the characteristic polynomial of the matrix $ B(T) $ is
\begin{eqnarray*}
	f(\lambda)=\lambda^4-(2t+ 2s)\lambda^3 -(5st + 14s+ 14t+ 20)\lambda^2 - (4st + 16s + 16t + 32)\lambda - 4s - 4t - 12.
\end{eqnarray*}
Since the partition is equitable, then by Lemma \ref{lem-eq}, each eigenvalue of the quotient matrix $ B(T) $ is also an eigenvalue of the distance matrix $ D(T) $. By Lemma \ref{-2}, we obtain that the value $-2$ is an eigenvalue of $ D(T) $ for any $ T\in \mathbb{T}_3 $ and its multiplicity
is at least $ s+t $. Clearly, $ |V(T)|= s+t+4$ and $-2$ is not a root of the characteristic polynomial $ f(\lambda) $ because of $ f(-2) =-12st - 12s - 12t - 12<0$. So the multiplicity of the eigenvalue $-2$ is $ s+t $. Besides, the other four distance eigenvalues of $ T$ are the four roots of the equation $	f(\lambda)=0$. If $-1$ is an eigenvalue of $ D(T) $, then
$$f(-1)=-st+1=0.$$
It follows that the tree with diameter 3 contains $ -1 $ as its distance eigenvalue when $t=s=1$, saying $ T\cong S_{2,2} $. \qed

We find that the trees obtained by successively adding an edge joining an endvertex of $ P_4 $ to a vertex in $S_{2,2}$ do not contain $-1$ as their distance eigenvalues. Referring to the proof of Theorem \ref{-1}, one reason is that the sum of components of $ \textbf{x} $ is not 0 where  $D(S_{2,2})\textbf{x}=-\textbf{x}$.

\end{document}